\documentclass[12pt,reqno]{article}

\usepackage[usenames]{color}
\usepackage{amssymb}
\usepackage{amsmath}
\usepackage{amsthm}
\usepackage{amsfonts}
\usepackage{amscd}
\usepackage{graphicx}

\usepackage[colorlinks=true,
linkcolor=webgreen,
filecolor=webbrown,
citecolor=webgreen]{hyperref}

\definecolor{webgreen}{rgb}{0,.5,0}
\definecolor{webbrown}{rgb}{.6,0,0}

\usepackage{color}
\usepackage{fullpage}
\usepackage{float}

\usepackage{breakurl}

\begin{document}

\theoremstyle{plain}
\newtheorem{theorem}{Theorem}
\newtheorem{corollary}[theorem]{Corollary}
\newtheorem{lemma}[theorem]{Lemma}
\newtheorem{proposition}[theorem]{Proposition}

\theoremstyle{definition}
\newtheorem{definition}[theorem]{Definition}
\newtheorem{example}[theorem]{Example}
\newtheorem{conjecture}[theorem]{Conjecture}

\theoremstyle{remark}
\newtheorem{remark}[theorem]{Remark}

\begin{center}
\vskip 1cm{\LARGE\bf The Fibonacci Sequence and Schreier-Zeckendorf Sets
\vskip 1cm}
\large
H\`ung Vi\d{\^e}t Chu\\
Department of Mathematics\\ 
University of Illinois at Urbana-Champaign \\
Champaign, IL 61820\\
USA \\
\href{mailto:hungchu2@illinois.edu}{\tt hungchu2@illinois.edu} \\
\end{center}
\vskip .2 in

\begin{abstract}
A finite subset of the natural numbers is \textit{weak-Schreier} if $\min
S \ge |S|$, \textit{strong-Schreier} if $\min S>|S|$, and \textit{maximal}
if $\min S = |S|$. Let $M_n$ be the number of weak-Schreier sets with
$n$ being the largest element and $(F_n)_{n\geq -1}$ denote the
Fibonacci sequence. A finite set is said to be Zeckendorf if it does
not contain two consecutive natural numbers. Let $E_n$ be the number of
Zeckendorf subsets of $\{1,2,\ldots,n\}$. It is well-known that $E_n =
F_{n+2}$. In this paper, we first show four other ways to generate the
Fibonacci sequence from counting Schreier sets. For example, let $C_n$
be the number of weak-Schreier subsets of $\{1,2,\ldots,n\}$. Then $C_n =
F_{n+2}$. To understand why $C_n = E_n$, we provide a bijective mapping to
prove the equality directly. Next, we prove linear recurrence relations
among the number of Schreier-Zeckendorf sets. Lastly, we discover the
Fibonacci sequence by counting the number of subsets of $\{1,2,\ldots,
n\}$ such that two consecutive elements
in increasing order always differ by an
odd number.
\end{abstract}

\section{Background and main results}
Let the Fibonacci sequence be $F_{-1} = 1$, $F_0 = 0$, and $F_m = F_{m-1}+F_{m-2}$ for all $m\ge 1$. We only concern ourselves with finite subsets of natural numbers greater than $0$ and use $\mathbb{N}$ for the set $\{1,2,3,\ldots\}$. We define a set to be 
\begin{itemize}
\item $\textit{weak-Schreier}$ if  $\min S \ge |S|$,
\item $\textit{strong-Schreier}$ if  $\min S>|S|$ and
\item $\textit{maximal}$ if $\min S = |S|$,
\end{itemize}
where $|S|$ is the cardinality of set $S$. Schreier sets are named after Schreier who defined them to solve a problem in Banach space theory in 1930 \cite{Schreier}. These sets were also independently discovered in combinatorics and are connected to Ramsey-type theorems for subsets of $\mathbb{N}$. For each $n\in \mathbb{N}$, let $M_n$ be the number of weak-Schreier sets with $n$ being the largest element. In notation,
$$M_n = | \{ S \subseteq \mathbb{N}: \min S \geq |S|
\text{ and } \max S = n\} |.$$
The first few values of $M_n$ are $1,1,2, 3, 5, 8, 13, 21, 34, \ldots$; indeed, Bird showed that $M_n = F_{n}$ for all $n$ \cite{Elec}. If we look at either  strong-Schreier sets or maximal sets instead, we can also generate the Fibonacci sequence. Let 
\begin{itemize}
    \item $A_n$ be the number of strong-Schreier sets $S$ with $\max S = n$,
    \item $B_n$ be the number of maximal sets $S$ with $\max S = n$,
    \item $C_n$ be the number of weak-Schreier subsets of $\{1,2,\ldots,n\}$ (including the empty set),
    \item $D_n$ be the number of strong-Schreier subsets of $\{1,2,\ldots,n\}$ (including the empty set).
\end{itemize}
For our sequence $(C_n)_{n\geq 1}$ and $(D_n)_{n\geq 1}$, we relax the condition about the maximum of our sets. Clearly, for each $n\in\mathbb{N}$, $M_n = A_n + B_n$, $C_n = \sum_{k=1}^n M_k + 1$ and $D_n = \sum_{k=1}^n A_n+1$. 
\begin{theorem} \label{4ways} For each $n\in\mathbb{N}$, we have
$A_{n} = F_{n-1}$, $B_n = F_{n-2}$, $C_n = F_{n+2}$ and $D_n = F_{n+1}$
\end{theorem}

The Fibonacci representation of natural numbers was first studied by Ostrowski \cite{Os} and Lekkerkerker \cite{Lek}.  In 1972, Zeckendorf proved that every positive integer can be uniquely written as a sum of non-consecutive Fibonacci numbers \cite{Z}. Since then, many papers have generalized this result and explored properties of the Zeckendorf decomposition: see \cite{burger, BBGILMT,BDEMMTW, DDKMMV, Ke, KKMW, Lek}. We instead focus on the important requirement for uniqueness of the Zeckendorf decomposition; that is, our set contains no two consecutive Fibonacci numbers. We give the same definition for natural numbers. 
\begin{definition}
A finite set of natural numbers is Zeckendorf if the set does not contain two consecutive natural numbers. 
\end{definition}
Let $E_n$ be the number of subsets of $\{1,2,\ldots,n\}$ that satisfy the Zeckendorf condition. It is well-known that $E_n = F_{n+2}$.

Two different ways of counting subsets of $\{1,2,\ldots,n\}$ give the same number; that is, $C_n = E_n$. To understand the connection, we construct a bijective mapping to show that $C_n = E_n$ directly. Our proof is independent of the fact that $C_n = E_n = F_{n+2}$ and thus, provides insight into the seemingly mysterious equality. 

\begin{theorem} \label{2mapping}
For each $n\in \mathbb{N}$,  $C_n = E_n$.
\end{theorem}

Next, a natural question is about sequences formed by the number of sets that satisfy both the Schreier and the Zeckendorf conditions. In particular, we say that a set satisfies the $k$-Zeckendorf condition if two arbitrary numbers in the set are at least $k$ apart. We discover linear recurrence relations among the number of sets satisfying both the Schreier and the $k$-Zeckendorf conditions. 

For each $n\in\mathbb{N}$, let $H_{k,n}$ be the number of subsets of $\{1,2,\ldots, n\}$  that 
\begin{itemize}
\item[(1)] satisfy the $k$-Zeckendorf condition;
\item[(2)] contain $n$; and
\item[(3)] are weak-Schreier.
\end{itemize}
\begin{theorem} \label{2cons}
Fix $k\in \mathbb{N}_{\ge 2}$. We have
$$
    H_{k,n} \ =\
\begin{cases} 1, & \mbox{ if } 1\le n\le k+1;\\
H_{k,n-1} + H_{k,n-(k+1)}, & \mbox{ if } n>k+1. 
\end{cases}$$
\end{theorem}

Using the exact same argument as in the proof of Theorem \ref{2cons},
we can also deduce the following theorems regarding strong and maximal
Schreier sets. For each $n\in\mathbb{N}$, let $I_{k,n}$ be the number
of subsets of $\{1,2,\ldots, n\}$  that (1) satisfy the $k$-Zeckendorf
condition, (2) contain $n$, and (3) are strong-Schreier.

\begin{theorem}
Fix $k\in \mathbb{N}_{\ge 2}$. We have
$$
    I_{k,n} \ =\ \begin{cases} 0, & \mbox{ if } n = 1;\\
    1, & \mbox{ if } 2\le n\le k+2;\\
    I_{k,n-1} + I_{k,n-(k+1)}, & \mbox{ if } n>k+2. 
\end{cases}$$
\end{theorem}
For each $n\in\mathbb{N}$, let $J_{k,n}$ be the number of subsets of
$\{1,2,\ldots, n\}$  that
\begin{itemize}
\item[(1)] satisfy the $k$-Zeckendorf condition;
\item[(2)] contain $n$; and
\item[(3)] are maximal.
\end{itemize}

\begin{theorem}
Fix $k\in \mathbb{N}_{\ge 2}$. We have
$$
    J_{k,n} \ =\ \begin{cases}  1, & \mbox{ if } n =1;\\
    0, &  \mbox{ if } 2\le n\le k+1;\\
   1, &  \mbox{ if } k+1 < n\le 2k+2;\\
J_{k,n-1} + J_{k,n-(k+1)}, & \mbox{ if } n>2k+2. 
\end{cases}$$
\end{theorem}

We give the following definition that is useful for the statement of our last result. 

\begin{definition}
Let $A = \{a_1,a_2,\ldots,a_k\}$ ($a_1<a_2<\cdots <a_k$) for some $k\in \mathbb{N}_{\ge 2}$. The difference set of $A$ is $\{a_2-a_1,a_3-a_2,\ldots,a_k-a_{k-1}\}$. The empty set and a set with exactly one element do not have a difference set. 
\end{definition}

We end with the following small result. 

\begin{theorem}\label{fibodd}
Fix $n\in \mathbb{N}$. The number of subsets of $\{1,2,\ldots,n\}$ 
\begin{enumerate}
\item that contain $n$ and whose difference sets contain only odd numbers is $F_{n+1}$,
\item whose difference sets contain only odd numbers (the empty set and sets with exactly one element vacuously satisfy this requirement) is $F_{n+3}-1$.
\end{enumerate}
\end{theorem}

\section{Proof of Theorem \ref{4ways}}
\begin{proof}[Proof of Theorem \ref{4ways}] We first prove item (1). Simple computation gives $A_1 = 0 = F_{0}$, $A_2 = 1 = F_{1}$, $A_3 = 1 = F_2$, $A_4 = 2= F_3$, and $A_5 = 3 = F_4$. It suffices to prove that $A_{n} + A_{n+1}= A_{n+2}$ for $n\ge 4$. Fix $n\ge 4$ and let us find a formula for $A_n$. The minimum number $k$ in our sets can take values from $1$ to $n$. For each value of $k$, there are $n-k-1$ numbers strictly between $k$ and $n$. Because our sets are strong-Schreier, they contain at most $k-3$ numbers out of these $n-k-1$ numbers. Hence, our formula for $A_n$ is 
$$A_n \ = \ \sum_{k=1}^{n-1}\sum_{j=0}^{k-3}\binom{n-k-1}{j}+1.$$
Note that the number $1$ in our formula accounts for the set $\{n\}$. It remains to show that $A_n + A_{n+1} = A_{n+2}$ or equivalently, $A_{n+2} - A_{n+1} = A_n$ for $n\ge 4$. We have

\begin{align*}
    A_{n+2} - A_{n+1} &\ =\ \sum_{k=1}^{n+1}\sum_{j=0}^{k-3}\binom{n-k+1}{j} - \sum_{k=1}^{n}\sum_{j=0}^{k-3}\binom{n-k}{j}\\
    &\ =\ \sum_{k=1}^{n}\sum_{j=0}^{k-3}\bigg(\binom{n-k+1}{j} - \binom{n-k}{j}\bigg) + \sum_{j=0}^{n-2}\binom{0}{j}\\
    &\ =\ \sum_{k=1}^n\sum_{j=1}^{k-3}\binom{n-k}{j-1} + 1.
\end{align*}
Therefore,
\begin{align*}
    A_{n+2}-A_{n+1}-A_n &\ =\  \sum_{k=1}^n\sum_{j=1}^{k-3}\binom{n-k}{j-1} - \sum_{k=1}^{n-1}\sum_{j=0}^{k-3}\binom{n-k-1}{j}\\
    &\ =\ \sum_{k=4}^n \sum_{j=1}^{k-3} \binom{n-k}{j-1} - \sum_{k=3}^{n-1} \sum_{j=0}^{k-3} \binom{n-k-1}{j} \ =\ 0.
\end{align*}
The last equality is because for each $4\le t\le n$, we have $\sum_{j=1}^{t-3} \binom{n-t}{j-1} = \sum_{j=0}^{(t-1)-3}\binom{n-(t-1)-1}{j}$. Hence, $A_{n+2} = A_{n+1}+A_n$ and we are done. 

Next, we prove item (2), which follows immediately from item (1). We know that 
$$B_n \ =\ M_n - A_n \ =\ F_n - F_{n-1} \ =\ F_{n-2}.$$

We prove item (3). Fix $n\ge 1$. We have
$$C_n = \sum_{k=1}^n M_k + 1 \ =\ \sum_{k=1}^n F_k +1 \ =\  (F_{n+2}-1)+1 \ =\ F_{n+2},$$
as desired. The number $1$ accounts for the empty set. The fact that $\sum_{k=1}^n F_k = F_{n+2}-1$ is due to Lucas \cite[p.\ 4]{Lucas}.

Similarly, we prove item (4). Fix $n\ge 1$. We have
$$D_n = \sum_{k=1}^n A_k + 1 \ =\ \sum_{k=1}^n F_{k-1}+1 \ =\ (F_{n+1}-1)+1 \ =\ F_{n+1}.$$

We complete our proof of Theorem \ref{4ways}.
\end{proof}
Let $L^{w}_n$ be the number of weak-Schreier sets as subsets of $\{1,2,\ldots,n\}$ with an even maximum. 
\begin{corollary}
For each $n\in\mathbb{N}$,  
    $$L^{w}_n \ =\ \begin{cases} F_{n}, & \mbox{ if } n \mbox{ is odd}; \\
     F_{n+1}, & \mbox{ if } n \mbox{ is even}.\end{cases}$$
\end{corollary}

\begin{proof} We have
\begin{align*}
    L^{w}_n \ = \ \sum_{{1\le k\le n} \atop{ 2\mid k}} M_k + 1 \ =
    \ \sum_{{1\le k \le n} \atop{ 2 \mid k}} F_k + 1.
\end{align*}
The number $1$ accounts for the empty set. 

If $n$ is even, $$L^{w}_n \ =\ \sum_{{1\le k\le n} \atop {2 \mid k}} F_k + 1 \ =\ (F_{n+1}-1)+1 \ =\ F_{n+1}.$$

If $n$ is odd, $$L^{w}_n \ =\ \sum_{{1\le k \le n} \atop { 2\mid k}} F_k + 1 \ =\ (F_n-1)+1 = F_n.$$
\end{proof}
Let $L^{s}_n$ be the number of strong-Schreier sets as subsets of $\{1,2,\ldots,n\}$ with an odd maximum. 
\begin{corollary}
For each $n\in\mathbb{N}$, 
    $$L^{s}_n \ =\ \begin{cases} F_{n}, & \mbox{ if } n \mbox{ is odd}; \\
    F_{n-1}, & \mbox{ if } n \mbox{ is even}.\end{cases}$$ 
\end{corollary}
\begin{proof} We have
\begin{align*}
    L^{s}_n \ = \ \sum_{{1\le k\le n} \atop {2\nmid k}} A_k + 1 \ =\ \sum_{1\le k \le n, 2\nmid k} F_{k-1} + 1.
\end{align*}

If $n$ is even, $$L^{s}_n \ =\ \sum_{{1\le k\le n} \atop {2\nmid k}} F_{k-1} + 1 \ =\ (F_{n-1}-1)+1 \ =\ F_{n-1}.$$

If $n$ is odd, $$L^{s}_n \ =\ \sum_{{1\le k \le n} \atop {2\nmid k}} F_{k-1} + 1 \ =\ (F_n-1)+1 = F_n.$$
\end{proof}

\section{Proof of Theorem \ref{2mapping} --- Explanation of the mysterious identity}
Recall that $C_n$ is the number of weak-Schreier sets as subsets of $\{1,2,\ldots,n\}$, while $E_n$ is the number of subsets of $\{1,2,\ldots,n\}$ that do not contain two consecutive numbers. At the first glance, $C_n$ and $E_n$ are little related, so it is surprising to see that $C_n = E_n$ for all $n\in \mathbb{N}$. 

For each $n\in \mathbb{N}$, let $X_n$ denote the set of weak-Schreier sets as subsets of $\{1,2,\ldots,n\}$ and let $Y_n$ denote the set of subsets of $\{1,2,\ldots,n\}$ that do not contain two consecutive numbers. In this section, we construct a bijective function $f: X_n\rightarrow Y_n$ to prove that $|X_n| = |Y_n|$. 

\begin{proof}[Proof of Theorem \ref{2mapping}]
Fix $n\in \mathbb{N}$. Let $A = \{a_1,a_2,\ldots,a_{k-1},a_k\}$ ($a_1<a_2<\cdots<a_k$) be a weak-Schreier subset of $\{1,2,\ldots,n\}$. Our mapping $f$ acts on $A$ as follows
$$f(A) = f(\{a_1,a_2,\ldots,a_{k-1},a_k\}) = \{a_1-(k-1), a_2-(k-2),\ldots, a_{k-1}-1, a_k\}.$$
Define $f(\emptyset) = \emptyset$. 
To show that $f$ is well-defined, we show that $\{a_1-(k-1),a_2-(k-2),\ldots, a_{k-1}-1,a_k\}$ is in $Y_n$. Because $A$ is weak-Schreier, $k\le a_1 < a_2< \cdots < a_k$. Hence, 
$$1\ \le\ a_1-(k-1) \ <\ a_2-(k-2) \ <\ \cdots \ <\ a_{k-1}-1 \ <\ a_k\ \le\ n.$$
Let $t_i = a_i - (k-i)$ for $1\le i\le k$. If $k = 1$, then $\{t_1\}$ is clearly in $Y_n$. If $k\ge 2$, then for each $2\le i\le k$, we have $$t_i-t_{i-1} \ =\ (a_i - (k-i)) - (a_{i-1}-(k-(i-1))) \ =\ (a_i - a_{i-1}) + 1 \ \ge\ 2.$$ Therefore, $\{t_1, t_2,\ldots, t_k\}\in Y_n$. So, $f$ is well-defined. 

Next, we prove that $f$ is injective. Suppose that $f(A) = f(B)$. Let $A = \{a_1,a_2,\ldots,a_k\}$ and $B = \{b_1,b_2,\ldots,b_k\}$, where $a_1 < a_2 <\cdots < a_k$ and $b_1 <b_2 < \cdots < b_k$. 
Because \begin{align*}a_1-(k-1) \ < \ a_2-(k-2) \ <\ \cdots \ < \ a_{k-1}-1 \ < \ a_k,\\
b_1-(k-1) \ < \ b_2-(k-2)\ < \ \cdots \ < \ b_{k-1}-1\ <\ b_k,\end{align*}
we know that $f(A) = f(B)$ implies $a_i - (k-i) = b_i - (k-i)$ for all $1\le i\le k$. Hence, $a_i = b_i$, which shows that $A =  B$. Therefore, $f$ is injective. 

Finally, we prove that $f$ is surjective. Let $C = \{c_1, c_2, \ldots, c_k\}\in Y_n$ be chosen, where $c_1<c_2<\cdots<c_k$. We claim that 
$$D \ =\ \{c_1+(k-1), c_2+(k-2), \ldots, c_{k-1}+1, c_k\}$$
satisfies $f(D) = C$ and $D\in X_n$. Because $C$ do not contain two consecutive numbers, we know that 
$$k\ \le\ c_1+(k-1) \ <\  c_2+(k-2) \ <\  \cdots \ <\  c_{k-1}+1 \ <\  c_k\ \le \ n.$$
Hence, $D\in X_n$.

We have shown that $f$ is both well-defined and bijective. Therefore, $|X| = |Y|$ or $C_n = E_n$, as desired. 
\end{proof}
\begin{remark}
We would like to discuss the motivation for the bijection $f$ used in the proof of Theorem \ref{2mapping}. Let $A$ be a Schreier set. The map $f$ serves to increase the gap between adjacent elements of $A$ by $1$, thus fulfilling the Zeckendorf condition that adjacent elements differ by at least 2. Furthermore, the weak-Schreier condition that $\min A\ge |A|$ ensures that the resulting set is in $\{1,2,\ldots, n\}$. 
\end{remark}

\section{Proof of Theorem \ref{2cons}}

Before we prove Theorem \ref{2cons}, we need a simple proposition.
\begin{proposition}\label{obs}
For $n,k\in\mathbb{Z}$, the following claims hold.
\begin{enumerate}
\item If $\big\lfloor \frac{n-2}{k+1}\big\rfloor = \big\lfloor \frac{n-k-2}{k+1}\big\rfloor$, then $\big\lfloor \frac{n-1}{k+1}\big\rfloor = \big\lfloor \frac{n-2}{k+1}\big\rfloor + 1$.
\item If $\big\lfloor \frac{n-2}{k+1}\big\rfloor > \big\lfloor \frac{n-k-2}{k+1}\big\rfloor$, then $\big\lfloor \frac{n-1}{k+1}\big\rfloor < \big\lfloor \frac{n-2}{k+1}\big\rfloor + 1$.
\item If $\big\lfloor \frac{n-k-2}{k+1}\big\rfloor = \big\lfloor \frac{n-2}{k+1}\big\rfloor$, then $\frac{n-k-2}{k+1}=\big\lfloor\frac{n-2}{k+1}\big\rfloor$.
\end{enumerate}
\end{proposition}

\begin{proof}
We prove claim (1). We have $$\bigg\lfloor \frac{n-2}{k+1}\bigg\rfloor \ =\ \bigg\lfloor \frac{n-k-2}{k+1}\bigg\rfloor \ =\ \bigg\lfloor \frac{n-1}{k+1}-1\bigg\rfloor \ =\ \bigg\lfloor \frac{n-1}{k+1}\bigg\rfloor-1.$$ Therefore, 
$$\bigg\lfloor\frac{n-1}{k+1}\bigg\rfloor \ =\ \bigg\lfloor\frac{n-2}{k+1}\bigg\rfloor+1.$$

Next, we prove claim (2). We have $$\bigg\lfloor \frac{n-2}{k+1}\bigg\rfloor \ >\ \bigg\lfloor \frac{n-k-2}{k+1}\bigg\rfloor \ =\ \bigg\lfloor \frac{n-1}{k+1}-1\bigg\rfloor \ =\ \bigg\lfloor \frac{n-1}{k+1}\bigg\rfloor-1.$$ Therefore, 
$$\bigg\lfloor\frac{n-1}{k+1}\bigg\rfloor \ <\ \bigg\lfloor\frac{n-2}{k+1}\bigg\rfloor+1.$$

Lastly, we prove claim (3). Write $n-k-2 = (k+1)p + q$ for some $0\le q\le k$. Then $$\frac{n-2}{k+1} \ =\ \frac{(k+1)p+q+k}{k+1} \ =\ p+\frac{q+k}{k+1} \ =\ p+1+\frac{q-1}{k+1}.$$ If $q\ge 1$, then $\big\lfloor \frac{n-2}{k+1}\big\rfloor = p+1 > p = \big\lfloor \frac{n-k-2}{k+1} \big\rfloor$, a contradiction. So, $q = 0$, implying that $\frac{n-k-2}{k+1} = \big\lfloor \frac{n-2}{k+1}\big\rfloor$.
\end{proof}
The following lemma is from \cite[Lemma 2.1]{KKMW} by Kologl\v{u} et al.
\begin{lemma}\label{numberofsols}
The number of solutions to $y_1 + \cdots + y_p = n$ with $y_i\ge c_i$ (each $c_i$ a non-negative integer) is $\binom{n-(c_1+\cdots+c_p)+p-1}{p-1}$.
\end{lemma}

\begin{proof}[Proof of Theorem \ref{2cons}]
Fix $k\ge 2$. We now find a formula for $H_{k,n}$ for all $n\in \mathbb{N}$. Fix $1\le \ell\le n-1$. Suppose that the set $\{a_1,\ldots,a_\ell,n\}$ satisfies all of our requirements. (For $\ell = 0$, we have the set $\{n\}$.) In particular, 
\begin{enumerate}
    \item $a_1\ge \ell+1$,
    \item $d_i = a_{i+1}-a_i\ge k$ and $d_\ell = n-a_\ell \ge k$.
\end{enumerate}
Note that \begin{align}\label{mainequa}a_1+\sum_{i=1}^\ell d_i \ =\ n.\end{align}
By Lemma \ref{numberofsols}, the number of sets satisfying Equation \eqref{mainequa}  is
$$\binom{n-(\ell+1+k\ell)+(\ell+1)-1}{(\ell+1)-1}\ =\ \binom{n-k\ell-1}{\ell}.$$
Therefore, the number of sets containing $n$ that are $k$-Zeckendorf and weak-Schreier is 

\begin{align*}
    H_{k,n} \ =\ \sum_{\ell=1}^{\big\lfloor \frac{n-1}{k+1}\big\rfloor}\binom{n-k\ell-1}{\ell} + 1.
\end{align*}
The number $1$ accounts for the set $\{n\}$ and we only let $\ell$ run up to $\big\lfloor \frac{n-1}{k+1}\big\rfloor$ to make sure that $n-k\ell-1\ge \ell$. It can be easily verified that $H_{k,n} = 1$ for $1\le n\le k+1$ because $\big\lfloor \frac{n-1}{k+1}\big\rfloor=0$ for $1\le n\le k+1$. It suffices to show that for $n\ge k+2$, $H_{k,n} = H_{k,n-1} + H_{k,n-(k+1)}$.
Equivalently,
\begin{align}\label{first}
    \sum_{\ell=1}^{\big\lfloor \frac{n-1}{k+1}\big\rfloor}\binom{n-k\ell-1}{\ell} \ =\ \sum_{\ell=1}^{\big\lfloor \frac{n-2}{k+1}\big\rfloor}\binom{n-k\ell-2}{\ell} + \sum_{\ell=1}^{\big\lfloor \frac{n-(k+1)-1}{k+1}\big\rfloor}\binom{n-k\ell-1-(k+1)}{\ell}+1.
\end{align}
Equivalently, noting that the $+1$ term cancels with the $l = 1$ term in the left hand side summation
\begin{align}\label{close}
    &\sum_{\ell=2}^{\big\lfloor\frac{n-2}{k+1}\big\rfloor}\bigg(\binom{n-k\ell-1}{\ell}-\binom{n-k\ell-2}{\ell}\bigg) + \sum_{\big\lfloor\frac{n-2}{k+1}\big\rfloor+1}^{\big\lfloor \frac{n-1}{k+1}\big\rfloor}\binom{n-k\ell-1}{\ell}\\\nonumber
    &\ =\ \sum_{\ell=1}^{\big\lfloor \frac{n-k-2}{k+1}\big\rfloor}\binom{n-k(\ell+1)-2}{\ell}.
\end{align}
We can simplify Equation \eqref{close}  further by applying the binomial coefficient recurrence
\begin{align*}
    \sum_{\ell=2}^{\big\lfloor\frac{n-2}{k+1}\big\rfloor}\binom{n-k\ell-2}{\ell-1} + \sum_{\big\lfloor\frac{n-2}{k+1}\big\rfloor+1}^{\big\lfloor \frac{n-1}{k+1}\big\rfloor}\binom{n-k\ell-1}{\ell}\ =\ \sum_{\ell=1}^{\big\lfloor \frac{n-k-2}{k+1}\big\rfloor}\binom{n-k(\ell+1)-2}{\ell}.
\end{align*}
Reindexing $\ell$ in the first summation, we have
\begin{align*}
    \sum_{\ell=1}^{\big\lfloor\frac{n-2}{k+1}\big\rfloor-1}\binom{n-k(\ell+1)-2}{\ell} + \sum_{\big\lfloor\frac{n-2}{k+1}\big\rfloor+1}^{\big\lfloor \frac{n-1}{k+1}\big\rfloor}\binom{n-k\ell-1}{\ell}\ =\ \sum_{\ell=1}^{\big\lfloor \frac{n-k-2}{k+1}\big\rfloor}\binom{n-k(\ell+1)-2}{\ell}.
\end{align*}
Subtract the first summation from both sides to have
\begin{align}\label{last}
 \sum_{\big\lfloor\frac{n-2}{k+1}\big\rfloor+1}^{\big\lfloor \frac{n-1}{k+1}\big\rfloor}\binom{n-k\ell-1}{\ell}\ =\ \sum_{\ell=\big\lfloor\frac{n-2}{k+1}\big\rfloor}^{\big\lfloor \frac{n-k-2}{k+1}\big\rfloor}\binom{n-k(\ell+1)-2}{\ell}.
\end{align}
We now prove that Equation \eqref{last}  is correct, which implies that Equation \eqref{first} is correct. 

\bigskip

\noindent Case 1: $\big\lfloor\frac{n-k-2}{k+1}\big\rfloor<\big\lfloor\frac{n-2}{k+1}\big\rfloor$. Then $\big\lfloor\frac{n-2}{k+1}\big\rfloor+1 > \big\lfloor \frac{n-1}{k+1}\big\rfloor$ by Proposition \ref{obs}. Therefore, two sides of Equation \eqref{last}  are identically $0$. 

\bigskip

\noindent Case 2: $\big\lfloor\frac{n-k-2}{k+1}\big\rfloor = \big\lfloor\frac{n-2}{k+1}\big\rfloor$. Then $\big\lfloor\frac{n-2}{k+1}\big\rfloor+1 = \big\lfloor \frac{n-1}{k+1}\big\rfloor$ and $\frac{n-k-2}{k+1} = \big\lfloor\frac{n-2}{k+1}\big\rfloor$ by Proposition \ref{obs}. Therefore, the left side of Equation \eqref{last}  is $$\binom{n-k(\big\lfloor\frac{n-2}{k+1}\big\rfloor+1)-1}{\big\lfloor\frac{n-2}{k+1}\big\rfloor+1} \ =\ 1$$ because $\frac{n-k-2}{k+1} = \big\lfloor\frac{n-2}{k+1}\big\rfloor$. Similarly, the right side is also equal to $1$. 

\bigskip

In both cases, Equation \eqref{last}  is correct. This completes our proof. 
\end{proof}

\section{Proof of Theorem \ref{fibodd}---A new way to generate the Fibonacci sequence}
\begin{proof}[Proof of Theorem \ref{fibodd}]First, we prove item (1). Let $P_n$ be the number of subsets of $\{1,2,\ldots, n\}$ that contain $n$ and whose difference sets contain only odd numbers.

\bigskip

\textit{Base cases:} For $n = 1$, we have $\{1\}$ to be the only subset of $\{1\}$ that satisfies our requirement. So, $P_1 = 1 = F_{2}$. For $n = 2$, we have $\{2\}$ and $\{1,2\}$ to be the only two subsets of $\{1,2\}$ that satisfy our requirement. So, $P_2 = 2 = F_{3}$. 

\bigskip

\textit{Inductive hypothesis:} Suppose that there exists $k\ge 2$ such that for all $n\le k$, $P_n = F_{n+1}$. 
We show that $P_{k+1} = F_{k+2}$. Let $O_n$ denote the set of subsets of $\{1,2,\ldots,n\}$ that satisfy our requirement. Observe that unioning a set in $O_{n-1-2i}$ (for $i\ge 0$) with $n$ produces a set in $O_n$ and any set in $O_n$ is of the form of a set in $O_{n-1-2i}$ plus the element $n$. Therefore, $$P_{k+1}\ =\ |O_{k+1}| \ =\ \sum_{{1 \leq i \leq k} \atop {2 \nmid i}}
|O_{k+1-i}| + 1\ =\ P_{k} + \sum_{{3 \leq i \leq k} \atop {2 \nmid i}} |O_{k+1-i}| + 1.$$ The number $1$ accounts for the set $\{n\}$.
If $k$ is odd, 
\begin{align*}
\sum_{{3 \leq i \leq k} \atop {2 \nmid i}} |O_{k+1-i}| &\ =\ |O_1| + |O_3| + \cdots + |O_{k-2}|\\
&\ =\ |F_2| + |F_4| + \cdots + |F_{k-1}| \ =\  F_k - 1 \ =\ P_{k-1} - 1.
\end{align*}
If $k$ is even,
\begin{align*}
\sum_{{3 \leq i \leq k} \atop {2 \nmid i}} |O_{k+1-i}| &\ =\ |O_2| + |O_4| + \cdots + |O_{k-2}|\\
&\ =\ |F_3| + |F_5| + \cdots + |F_{k-1}| \ =\  F_k - 1 \ =\ P_{k-1}-1.
\end{align*}
In both cases, we have $\sum_{{3 \leq i \leq k} \atop {2 \nmid i}} |O_{k+1-i}| = P_{k-1} - 1$. Therefore, $P_{k+1} = P_k + P_{k-1}= F_{k+1}+F_{k}=F_{k+2}$, as desired.

Next, we prove item (2). Let $Q_n$ be the number of subsets of
$\{1,2,\ldots, n\}$ whose difference sets contain only odd numbers
is $Q_n$ (the empty set and sets with exactly one element vacuously
satisfy this requirement). Note that by definition of $P_n$ and $Q_n$,
we have $$Q_n \ =\ 1+ \sum_{k=1}^n |P_k| \ =\ 1+ \sum_{k=1}^n F_{k+1} \
=\ \sum_{k=1}^{n+1}F_k \ =\ F_{n+3} - 1,$$ as desired. (The $+1$ before
the first summation accounts for the empty set.)  
\end{proof}

\section{Acknowledgments}
The author would like to thank the anonymous referee and the editor for various helpful comments that clarify several points made in this paper. Thanks to Kevin Beanland at Washington and Lee University for useful comments on earlier drafts.

\bigskip
\hrule
\bigskip

\noindent 2010 {\it Mathematics Subject Classification}: 11B39.

\noindent \emph{Keywords:} Fibonacci sequence, linear recurrence, combinatorics, Schreier set.

\end{document}